\documentclass[11pt,a4paper,twoside]{article}
\usepackage{amsthm,amsfonts,amsmath,amscd,amssymb}
\usepackage{latexsym}
\usepackage{euscript}
\usepackage{enumitem}
\usepackage{graphicx}
\usepackage{tikz}
\usepackage{caption}
\usepackage{subcaption}
\usepackage{cite}
\usepackage[utf8]{inputenc}
\usepackage[english]{babel}
\usepackage{xcolor}

\usepackage{jmpag} 

 \def \v {\vskip 0.2cm }
 \def \V {\vskip 0.5cm} 
 
\def \bA {\breve A}

\def \tA{\tilde A}

\def \bx {{\bf x}}

\def \bv {{\bf v}}

\def \G {\Gamma}
\def \g {\gamma}

\def \a {\alpha}

\def \b {\beta}

\def \k {\kappa}

\def \de {\delta}

\def \D {\Delta}

\def \vep {\varepsilon}

\def \s {\sigma}

\def \vp {\varphi}

\def \l {\lambda}

\def \n {\noindent}

\def \Tr {{\mathrm{\, Tr}}}

\def \bC {{\mathbb C}}

\def \bE {{\mathbb E}}

\def \bR {{\mathbb R}}

\def \CA {{\cal A}}

\def \CE {{\cal E}}

\def \CN {{\cal N}}

\def \CD {{\cal D}}

\def \D {\Delta}

\def \hD {\hat \Delta}
\def \a {\alpha}

\def \U {\Upsilon} 

\def \l {\lambda}

\begin{document}



\title{Asymptotic absence of poles of Ihara zeta function
of large Erd\H os-R\'enyi random graphs}



\author{Oleksiy Khorunzhiy}
\address{Universit\'e de Versailles - Saint-Quentin, 45, Avenue des Etats-Unis, Versailles 78035, France}
\email{oleksiy.khorunzhiy@uvsq.fr}



\BeginPaper 



\newcommand{\ep}{\varepsilon}
\newcommand{\eps}[1]{{#1}_{\varepsilon}}



\begin{abstract}
Using recent results on the concentration of the largest  eigenvalue and maximal vertex degree
of large random graphs, we show  
that the infinite sequence of Erd\H os-R\'enyi random graphs $G(n,\rho_n/n)$ such that 
$\rho_n/\log n$ infinitely increases as $n\to\infty$ verifies
a version of the graph theory Riemann Hypothesis.
\key{random graphs, random matrices, Ihara zeta function, graph theory Riemann hypothesis}
\end{abstract}


\section{Ihara zeta function and graph theory Riemann hypothesis}\label{S1}

Given a finite connected non-oriented graph 
$\G = (V,E)$ with the vertex set $V= (\a_1, \dots, \a_n)$ and the edge set 
$E$, 
the Ihara zeta function (IZF) $Z_\G (u)$ is determined for sufficiently small $|u|$ by 
equality
\begin{equation}\label{(1.1)}
Z_\G(u) = \prod_{ [C]} \left(1 - u^{\nu(C)}\right)^{-1},
\end{equation}
where $[C]$ denotes the equivalence class of closed primitive backtrackless tailless paths $C$ and $\nu(C)= k-1$, $k$ being the length of $C$ \cite{I}. 
The $k$-step path over graph 
$C= (\a_{i_1}, \a_{i_2}, \dots, \a_{i_{k-1}}, \a_{i_k})$,
$\{\a_{i_l}, \a_{i_{l+1}}\} \in E$ is closed  when $\a_{i_k}=\a_{i_1}$. 
The path $C$ is backtrackless if $\a_{i_{l-1}}\neq \a_{i_{l+1}}$ for all
$l=2, \dots, k-1$. The path $C$ is tailless if 
$\a_{i_2}\neq \a_{i_{k-1}}$.
The equivalence class $[C]$ includes $C$ and all paths obtained from $C$ with the help of all cyclic 
permutation of its elements. 
The closed path $C$ is primitive if  	there is no smaller path
$\tilde C$ such that $C= \tilde C^k$.

Zeta function  \eqref{(1.1)} has been introduced
by Y. Ihara in the algebraic context \cite{I}. 
Ihara's theorem says that the IZF  \eqref{(1.1)} is the reciprocal of a polynomial and that for sufficiently small
$\vert u\vert$
\begin{equation}\label{(1.2)}
Z_\G(u)^{-1} = (1-u^2)^{r-1} 
\det\big( I + u^2(B-I) - uA\big), \quad u\in \bC, 
\end{equation}
where $A= (a_{ij})_{i,j=1, \dots, N}$ is the adjacency matrix of 
$\G$, $B= \hbox{diag}\left(\sum_{j=1}^N a_{ij}\right)$ and 
\begin{equation}\label{(1.3)}
r-1= \Tr (B-2I)/2.
\end{equation}
Let us note that the Ihara zeta function can also be determined as the exponential expression
\cite{Te},
\begin{equation}\label{(1.4)}
Z_\G (u) = \exp\left\{ \sum_{k\ge 1} {\CN_k\over k} u^k\right\},
\end{equation}
where $\CN_k$ is the number of classes of closed backtrackless tailless 
primitive 
paths of the length $k$ over the edges of $\G$. 
The Ihara's theorem has been proven initially for 
$q+1$-regular graphs, then it has been generalized by Bass
to the cases of  possibly irregular graphs \cite{Ba} (see also \cite{FZ}).

There exists an analog of the Riemann hypothesis formulated 
for $q+1$-regular graphs $\G=X^{(q+1)}$ with the help of the Ihara zeta function. According to the definition by Stark and Terras \cite{ST}, 
a graph $X^{(q+1)}$ verifies 
the  {\it graph theory
Riemann hypothesis} (GTRH)
iff its Ihara zeta function is such that 
\begin{equation}\label{(1.5)}
\hbox{Re\,}s\in (0,1) \ {\hbox{and }}\left( Z_{X^{(q+1)}}\big(q^{-s}\big)\right)^{-1}=0 \ \ \ \hbox{ imply } \ 
\hbox{Re\,} s = {1\over 2}.
\end{equation}
This relation means that the graph $X^{(q+1)}$ is such that 
there is no poles of the Ihara zeta function $Z_X(u)$ in the disk  
$1/q< |u|< 1$ excepting those situated on the circle $|u|= {1/\sqrt q}$. 
The following statement is formulated 
by Stark and Terras \cite{ST} as a corollary of the formula \eqref{(1.2)}.

\v 
\begin{lemma}{\label{L:ST}}  A finite $(q+1)$-regular graph $X^{(q+1)}$ satisfies the Riemann hypothesis iff for every eigenvalue $\l$ of its adjacency matrix 
$A_X$, we have
\begin{equation}\label{(1.6)}
\vert \l\vert \neq q+1 \ \ \ \hbox{implies} \ \ \ 
\vert \l \vert \le 2 
\sqrt q.
\end{equation}
\end{lemma}

We reproduce  the proof of this lemma  in Section \ref{S4} of the present paper.
It is fairly simple and uses an elementary  observation that 
in the case of $q+1$-regular graphs, relations 
\eqref{(1.2)} and \eqref{(1.5)} reduce the problem 
to the study of  zeroes of the  quadratic 
equation 
$$
1+qu^2 - \l u= 0,
$$ 
whose discriminant
is negative if and only if $| \l |<2\sqrt q$.

 \v 
Relation   \eqref{(1.5)} can be reformulated
in more convenient for us form  as follows:

\noindent for any  complex $v\in D_q$, 
\begin{equation}\label{(1.7)}
D_q = \left\{ z\in \bC: {1\over \sqrt q}< \vert z \vert <\sqrt q\right\}, 
\end{equation}
 the following statement 
\begin{equation}\label{(1.8)}	
 \left(Z_{X^{(q+1)}} \left({v\over \sqrt q}\right)\right)^{-1} = 0 
  \quad \hbox{implies} \quad 
\vert v\vert =1
\end{equation}
is true. 
This means that the regular graph $X^{(q+1)}$ verifies the GTRH
 if and only if the function 
\begin{equation}\label{(1.9)}
\Phi^{(q+1)}(v)= 
Z_{X^{(q+1)}}
\left({v\over \sqrt q}\right)
\end{equation}
has no poles in the region $v\in D_1^{(q)} \cup D_2^{(q)}$, where
\begin{equation}\label{(1.10)}
D_1^{(q)}= \left\{ v\in \bC: 1/\sqrt q<|v|<1\right\} \quad {\hbox{and}}
\quad D_2^{(q)} = \left\{v \in \bC: 1<|v|< \sqrt q\right\}.
\end{equation}

Relation \eqref{(1.6)} represents a widely known in graph theory and applications the second eigenvalue conjecture
(see \cite{A} and references therein).   This condition means that the distance
 between the maximal and the second maximal in absolute value eigenvalues
of a  $q+1$-regular graph is greater than $q+1-2\sqrt q$. 
The $(q+1)$-regular graphs that satisfy condition \eqref{(1.6)} are 
determined   by  Lubotzky, Phillips and Sarnak 
 as  the {\it Ramanujan graphs} \cite{LPS} (see also the review \cite{Mur}
 and references therein).   The Ramanujan  graphs are known to be good expanders that make good communication networks (see e.g. \cite{Fr,Mur}). This property can be explained by the observation  that the diameter of a $(q+1)$-regular graph  
 is minimized by minimizing the second maximal eigenvalue 
 because the maximal eigenvalue is always equal to  $q+1$ \cite{Chung,Quen,Mur}.
 It is proved by  Friedman \cite{Fr-1,Fr-2} 
 the proportion of $q+1$-regular graphs with $n$ vertices such that \eqref{(1.6)} is true with
 $2\sqrt q$ replaced by $2\sqrt q+ \vep$ goes to 1 as $n\to\infty$ 
for any $\vep>0$.

\v 
The notion of the Ramanujan graphs mostly concerns the regular graphs. 
 Several extensions of this notion has been considered \cite{L,ST}.
 A definition of the Ramanujan graphs in the case of 
 non-regular graphs has been proposed by Lubotzky \cite{L}. It   is given in terms of the eigenvalues of the adjacency matrix of a graph $\G$, its
spectral radius and the spectral radius of the adjacency operator on  the universal covering tree of $\G$.

 The case when $\G$ is chosen at random from the set of 
 all possible graphs of $n$ vertices has been considered in \cite{K-17}.
 More precisely, the eigenvalue distribution of the matrix 
 $I+H(u)= I + u^2(B-I) - uA$ has been studied, where 
 $A$ is the adjacency matrix of the  Ed\H os-R\'enyi random graphs $G(n,\rho_n/n)$ 
 (see the next section for the rigorous definition of the ensemble $G(n,p)$).
 It is shown that  in the limit $n\to\infty$ and $\rho_n /\log n\to\infty$ 
  the limiting eigenvalue distribution of $H(u)$
  with properly normalized parameter $u= v/\sqrt {\rho_n}$ 
  exists and its density is given by a shift of the Wigner semi-circle distribution. Then one can show that 
 the limit of the mean value of 
  ${1\over n} \log Z_{\G } (v/\sqrt{\rho_n})$, if it exists, 
  satisfies  a  version of  \eqref{(1.5)}.

Another approach allowing   
to include the case of  non-regular random  graphs into consideration  
is based on the following representation of zeta function \eqref{(1.1)},
\begin{equation}\label{(1.11)}
Z_\Gamma(u)^{-1}  = \det (I- uW_\Gamma),
\end{equation}
where $W_{\Gamma}$ is the {\it non-backtracking matrix} 
of $\Gamma$ (see \cite{Hashi,ST} and references therein). 
Representation \eqref{(1.11)} is known as the Ihara-Bass formula and can be 
serve  as the basis of the proof of the Ihara theorem \eqref{(1.2)} for connected graphs.

Regarding representation \eqref{(1.11)}, Stark and Terras defined $\Gamma$ to satisfy the graph theory Riemann hypothesis if the matrix $W_\Gamma $ has no eigenvalues 
with the absolute values inside of the interval  $(\sqrt{r_{W_\Gamma}}, r_{W_\G })$,
where $r_{W_\G }$ is the Perron-Frobenius eigenvalue of $W_\G $.
In paper \cite{BLM}, the spectral properties of the non-backtracking matrix $W_\G$ of the
 the  Erd\H os-R\'enyi random graphs $\Gamma \in G(n, p_n)$ have been studied.
 It is shown that the graphs $\G \in G(n,\alpha/n)$ 
verify a weak Ramanujan property in the sense that in the limit
$n\to\infty$, they satisfy with high probability the graph theory Riemann 
hypothesis formulated  above
\cite{ HST}. More precisely,   it is proved that 
$r_{W_\G}\sim \a$ and all other eigenvalues $\lambda$ of $W_\G$  verify
$| \lambda | \le \sqrt \a + o(1)$ with high probability as $n\to\infty$. 
It is stated  that in this sense, the Erd\H os-R\'enyi random graphs $G(n,\a/n)$ asymptotically satisfy the graph theory Riemann hypothesis
\cite{BLM}. In this work the term "with high probability" 
to indicate the situation when 
 one or another statement is true with probability 
$1+o(1)$, $n\to\infty$ i.e. tending to 1 as $n$ infinitely increases. 
This shows that  the results 
of \cite{BLM} are   much in the spirit of statements by Friedman cited above. 
 It should be also noted that the results of \cite{BLM} 
as well as those of \cite{Fr-1, Fr-2} are obtained with the help of the 
 study of moments of 
non-backtracking matrix $W_\G$.

The limiting eigenvalue distribution of non-backtracking matrices $W_\Gamma$
of Erd\H os-R\'enyi random graphs 
$G(n,p)$ have been also studied 
in the asymptotic regimes when either $p\approx Const$ or $p=o(1)$ as 
$n\to\infty$ \cite{WW}. More fine spectral characteristics
of $W_\G$ such as the presence of isolated eigenvalues inside and outside
of the bulk of the spectrum of $W_\G$ have been considered 
in \cite{CZ} for a  generalization 
of the Erd\H os-R\'enyi random graphs $G(n,\alpha /n)$
in the case when $\alpha/\log n\to \infty$ as $n\to\infty$. 
It is shown, in particular, that 
with probability $1-o(1)$ all eigenvalues of $W_\G$ are located 
on the distance $o(\sqrt\alpha)$ of a circle of radius $\sqrt{\a-1}$
excepting two ones that are close to $1$ and $\a$ as $n\to\infty$. 
Since eigenvalues of $W_\G$ determine uniquely the poles of $Z_\G(u)$ \eqref{(1.11)},
the results of \cite{CZ} can be interpreted in the sense that 
the proportion of Erd\H os-R\'enyi random graphs 
that verify the graph theory Riemann hypothesis \eqref{(1.5)} tends to one
as $n\to\infty$. This formulation put 
 \cite{CZ} in line with the works \cite{Fr-1,Fr-2} and \cite{BLM}.
 The results of \cite{CZ} 
 are obtained on the base of the known facts from spectral properties
 of the adjacency matrices $A_\G $ of random graphs combined with the 
 concentration results on the elements of $B$ and   perturbation theorems by  Bauer and Fike allowing one to study the spectrum of 
 the non-backtracking matrix $W_\G$
 on the base of the knowledge of that of 
 the adjacency matrix $A_\G$ (see also \cite{BLM}).

In this paper, we follow the approach of \cite{K-17}
based on the  study of the spectrum of 
$I+u^2(B-I) - uA$ of the right-hand side of \eqref{(1.2)}. 
This method  seems to be more simple and transparent than  that using  the 
non-backtracking matrix $W_\G$. 
Using  classical perturbation theorems of the Weyl type
for  singular values of matrices,  
 concentration properties of
$B$ \cite{K-20} and  recent results on the concentration 
of the maximal eigenvalue of $A_\G$ \cite{LMZ}, we prove a statement 
that can be regarded as an improvement of that  of   \cite{CZ}.  
Namely, we show  that the 
infinite series of probabilities 
of the events
that  
 the  normalized zeta function $Z_\G(v/\sqrt{\rho_n})$
 of the Erd\H os-R\'enyi random graphs $G(n,\rho_n/n)$ in the asymptotic regime when 
$\rho_n\gg \log n$ has a pole in any domain close to 
$D=\{v\in \bC: |v|\neq 1\}$
converges.
The result obtained in the present paper can be regarded as a  
one more confirmation of the conjecture  that 
almost all   Erd\H os-R\'enyi random graphs 
$\{G(n,p_n)\}_{n\ge 1}$ satisfy, in the limit  $n\to\infty$, $n p_n/\log n\to\infty$, a version of the graph theory Riemann hypothesis.

\section{Ihara zeta function of  Erd\H os-R\'enyi random graphs}\label{S2}

Let us consider a family of
jointly independent random variables 
$$
{\CA}_{n,\rho} = \{ a_{ij}^{(n,\rho)}, \ 1 \le i< j\le n\}
$$
that have the probability distribution
$$
a_{ij}^{(n,\rho)}= 
 \begin{cases}
  1 , & \text{with probability $p_n=\rho/n$} , \\
0, & \text {with probability $1-{\rho/n}\, $} 
\end{cases}, \   \  0<\rho<n.
$$ 
We assume that the  family $\CA_{n,\rho}$ is determined
on a  probability space $\Omega_{n,\rho}$
and denote by $\bE = \bE_{n,\rho}$ the mathematical expectation
with respect to the probability measure $P=P_{n,\rho}$ generated by 
$\CA_{n,\rho}$. 

The ensemble of real symmetric random matrices 
$
A^{(n,\rho)}$ with elements 
\begin{equation}\label{(2.1)}
\left(A^{(n,\rho)}\right)_{ij}=
\begin{cases}
  a^{(n,\rho)}_{ij} , & \text{if $i< j$} , \\
a^{(n,\rho)}_{ji}, & \text {if  $i>j $}, \\
0 , & \text{if $i=j$}, 
\end{cases}
\quad i,j \in \{1, \dots, n\}
\end{equation}
 can be regarded as the adjacency matrix 
 of a  non-oriented random graph $\G^{(n,\rho)}$. 
 The family of such random graphs $\{\G^{(n,\rho)}\}$ is  usually denoted by $G(n,p_n)$, where we have taken $p_n= \rho/n$.
 This family 
 is equivalent in many aspects 
to the ensemble  of random Erd\H os-R\'enyi graphs \cite{ER} and is often
referred simply as the  Erd\H os-R\'enyi random graphs (see monograph \cite{B}).

\v Given $\G^{(n,\rho)}$, we consider the corresponding right-hand side of \eqref{(1.2)}
an say that it determines the Ihara zeta function of $\G^{(n,\rho)}$
despite of the fact that this  graph can be disconnected,
\begin{equation}\label{(2.2)}
(1-u^2)^{r-1} \, \det\big( I + u^2(B^{(n,\rho)}-I) - uA^{(n,\rho)}\big)
= \left(Z_{\G^{(n,\rho)}}(u)\right)^{-1}, \quad u\in \bC,
\end{equation}
where 
\begin{equation}\label{(2.3)}
\left(B^{(n,\rho)}\right)_{ij} =\,  \delta_{ij}\  \sum_{k=1}^n \,  a^{(n,\rho)}_{ik}.
\end{equation}
According to \eqref{(1.3)}, we have 
denoted in \eqref{(2.2)} 
\begin{equation}\label{(2.4)}
r=\, {1\over 2} \,  \sum_{i,j=1}^n a^{(n,\rho)}_{ij} \, - n+1.
\end{equation}

We study IZF \eqref{(2.2)} in the limiting transition $n\to\infty$ when 
the average value of the vertex degree of $\G^{(n,\rho_n)}$ given by 
$(B^{(n,\rho_n)})_{ii}$ \eqref{(2.3)}
 goes to infinity
more rapidly than  $\log n$. This means that 
\begin{equation}\label{(2.5)}
n\to\infty,  \quad  \rho_n/\log n = \chi_n \to \infty.
\end{equation}
We denote this limiting transition by $(\rho, \chi)_n\to\infty$. 
Writing \eqref{(2.5)}, we assume that an infinite sequence $(\chi_n)_{n\ge 1}$ is determined
and $\rho_n$ is given by relation 
$$
\rho_n = \chi_n \log n, \quad n\ge 1.
$$
In what follows, we omit the subscript in $\rho_n$ when no confusion can arise. 

\v

In paper \cite{K-17}, it is shown that
in the limit \eqref{(2.5)}, it is natural to consider \eqref{(2.4)} 
with the spectral parameter $u$ normalized 
by the square root of $\rho= \rho_n$.
Thereby we introduce  a normalized version of Ihara zeta function \eqref{(2.2)} with 
the spectral parameter $u= v/\sqrt \rho$
$$
\tilde Z^{(n,\rho)}(v, \omega)= Z_{\G^{(n,\rho)}(\omega)}\left({v\over \sqrt \rho}\right), 
\quad \omega \in \Omega_{n,\rho_n}, \quad v\in \bC, 
$$
where 
\begin{equation}\label{(2.6)}
\left( Z_{\G^{(n,\rho)}}\left({v\over \sqrt \rho}\right)\right)^{-1} =
\left(1-{v^2\over \rho}\right)^{r-1} 
\det\left( -v H^{(n,\rho)}\right),
\end{equation}
where 
$$
H^{(n,\rho)} (v) =  
I+ {v^2\over \rho} \left( B^{(n,\rho)}-I\right) - 
{1 \over \sqrt \rho}  A^{(n,\rho)}. 
$$

Our main result is given by the following statement.

\begin{theorem}\label{T1}

Let  $D_{\vep, \vep'}$ with $\vep, \vep'>0$ be a union of two 
complex domains 
\begin{equation}\label{(2.7)}
D^{(1)}_{\vep, \vep'} = \{ z\in \bC: \vep' < |z |< 1-\vep\} \quad {\hbox{and}} \quad 
D^{(2)}_{\vep, \vep'} = \{ z\in \bC: 1+\vep < |z |< 1/\vep'\} . 
\end{equation}
We consider two subsets $\Phi^{(1)}_{n,\rho}\left(\vep, \vep'\right)$
and $\Phi^{(2)}_{n,\rho}\left(\vep, \vep'\right)$
determined by relation
\begin{equation}\label{(2.8)}
\Phi^{(i)}_{n,\rho}\left(\vep, \vep'\right)
= \left\{\omega \in \Omega: \,  {\hbox{there exists}}\  v\in D^{(i)}_{\vep, \vep'} \ {\hbox{such that }}  \  
\left(\tilde Z^{(n,\rho)}(v,\omega)\right)^{-1}=0 
\right\},
\end{equation}
for $i=1$ and $i=2$.
With the choice of 
\begin{equation}\label{(2.9)}
\vep_n = {2 \over (\chi_n)^{1/8}} \quad {\hbox{and }}
\quad  \vep '_n = {1\over 
\sqrt{ \rho_n} (1-\k)},\quad \kappa>0,
\end{equation}
the following series of probabilities converge,
\begin{equation}\label{(2.10)}
\sum_{n\ge 1}  P_{n,\rho_n}
\left(\Phi^{(i)}_{n,\rho_n}
\left(\vep_n, \vep'_n\right)
\right) <\infty, \quad i=1,2
\end{equation}
under asymptotic condition \eqref{(2.5)}.	
\end{theorem}

We prove Theorem \ref{T1}  in Section \ref{S3}  below. Let us note that we can prove 
slightly more powerful statement by adding to the complex domains 
$D^{(i)}_{\vep, \vep'}$ (\ref{(2.7)})  real intervals 
$$
I^{(1)}_{\hat \vep, \vep'}= \{v \in \bR: \vep'<v<1-\hat \vep\}
\quad {\hbox{and}}\quad 
I^{(2)}_{\hat \vep, \vep'}= \{v \in \bR: 1+\hat \vep<v<1/ \vep'\},
 $$
where $\vep'= \vep'_n$ is given by (\ref{(2.9)}) and 
\begin{equation}\label{(2.11)}
\hat \vep_n = {\hat h\over (\chi_n)^{1/4}}
\end{equation}
 with sufficiently large $\hat h>0$.  We concentrate ourself on the main case of complex domains, the case of real $v$ is briefly discussed at the end of Section \ref{S3}.
 
\v
Let us formulate a corollary of Theorem \ref{T1}  that characterizes the points of the complex plane in the form close to \eqref{(1.10)}. 
 \v 
\begin{corollary} \label{C:2.1} For any given 
$$
v\in D^{(1)}\cup D^{(2)}= \{ z: z \in \bC, 0< | z|<1\} \cup \{ z: z\in \bC, 1< | z|\},
$$
the following series of probabilities converges,
\begin{equation}\label{(2.12)}
\sum_{n\ge 1}P_{n,\rho_n}
\left( \left\{ \omega\in \Omega_{n,\rho_n}:  
 \left( \tilde Z^{(n,\rho_n)}(v, \omega)
 \right)^{-1} = 0 \right\} \right)< \infty,
\end{equation}
 where $\rho_n = \chi_n \log n$, $\chi_n \to \infty$ as $n\to\infty$.
\end{corollary}

The proof of this statement follows immediately from the proof of Theorem \ref{T1}.

\V 
Let us note that the statement of Corollary \ref{C:2.1}  can be formulated 
 in its equivalent form,
 \begin{equation}\label{(2.13)}
 \sum_{n\ge 1} P_{n, \rho_n}
 \left( 
 \left\{ 
 \omega: 
 \left( Z_{\G^{(n,\rho_n)}}\left({v\over \sqrt \rho_n}\right)\right)^{-1}
  = 0 \right\} \right)=+\infty 
 \quad {\hbox{implies}} \quad |v|=1. 
\end{equation}
This shows that Corollary \ref{C:2.1}  can be regarded as a direct  analog of the statement \eqref{(1.8)} in the case of large random graphs.

Let us outline the proof of Theorem \ref{T1}. Definition  
\eqref{(2.6)} shows that 
if $v^2\neq \rho$, then the function $Z_{\G^{(n,\rho)}}$ has a pole at $v$
if and only if 
\begin{equation}\label{(2.14)}
v\left(1-{1\over \rho}\right) + v  -   \lambda_j
\left({1\over \sqrt \rho} A + v\left( {1\over \rho} B - I\right)\right) =0
\end{equation}
for some $j$, where $\lambda_j(M)$ denotes the $j$-th eigenvalue of $M$. 
If one accepts that the expression in last braces  asymptotically vanishes
\begin{equation}\label{(2.15)}
\Vert  {1\over \rho} B - I\Vert = o(1),
\end{equation}
then it can  be regarded as a 
 small perturbation of the eigenvalues 
 of $\tilde A = A/\sqrt \rho$. 
 Neglecting the corresponding term in the right-hand side of 
 \eqref{(2.14)} as well as the \mbox{vanishing} term $-1/\rho$ in the first braces of \eqref{(2.14)}, one could say 
 that 
 \eqref{(2.14)}
 is \mbox{equivalent} to the condition that 
 \begin{equation}\label{(2.16)}
 {1\over v} + v - \lambda_j(\tilde A)= 0 
 \end{equation}
 for some $j\in \{1, \dots, n\}.$
It is known since  \cite{FK,KS} that the normalized adjacency matrices 
of the Erd\H os-R\'enyi random graphs $G(n,\rho/n)$
have all  eigenvalues,  excepting the maximal one 
 $\lambda_1(\tilde A)$, 
asymptotically bounded in absolute value by $2+ \de_n$
in the limit $n\to\infty$, $\rho\gg \log n$, where $\de_n$ tends to zero.
 Therefore  the probability of the event  \eqref{(2.16)} 
 rapidly decays  for  all $j\in \{2, \dots, n\} $ as $n\to\infty$ for any $v$ verifying  
 $| v | \neq 1$.  It is worthy note that important part of 
 this proposition 
 is that it can be proved uniformly with respect to complex $v$ belonging to 
 growing domains $D^{(1)}$ and $D^{(2)}$ \eqref{(2.7)}. 
 To do this, we 
 use an important  property of the ellipsoidal curves of the form
 \begin{equation}\label{(2.17)}
 \CE_r= \left\{w(z)=   z+ {1\over z}, \ z = r e^{i \vp}, \ 0\le \vp< 2\pi
 \right\}
 \end{equation}
 and the distance between the points of real axis and $\CE_r$. 
 
 To establish convergence of the  series \eqref{(2.10)},
we use recent results on the concentration properties of eigenvalues of adjacency matrices of random graphs \cite{BGBK,LMZ}.
Relation \eqref{(2.15)}  reflects the concentration property of  diagonal elements of $B-I$
\begin{equation}\label{(2.18)}
\max_{ 1\le i\le n} \left| {1\over \rho} \left( B^{(n,\rho)}\right)_{ii} - 1\right|= o(1), \quad n, \rho\to\infty
\end{equation}
that is also  well known in the literature 
  (for  example, see monograph \cite{B} and paper \cite{K-20}). 
Convergence \eqref{(2.18)} can be interpreted as asymptotic regularity
  of the Erd\H os-R\'enyi random graphs when the average vertex degree
  goes to infinity. 
  Taking into account this observation, one 
  can say that the proof of Theorem \ref{T1}  is equivalent in certain sense to 
  the 
  proof of  a stochastic version of  Lemma \ref{L:ST} by Stark and Terras.

\section{Proof of Theorem \ref{T1}}\label{S3}

Let us rewrite definition \eqref{(2.7)} in the form 
\begin{equation}\label{(3.1)}
H^{(n,\rho_n)}(v) = \tilde A^{(n,\rho_n)}   - \g^{(n,\rho_n)}(v) I 
- v \hat B^{(n,\rho_n)},
\end{equation}
where 
$
\hat B  = 
\tilde B  -  I (n-1)/ n
$,
$$
\left(\tilde B^{(n,\rho)}\right)_{ij} = 
{1\over \rho} \left(B^{(n,\rho)}\right)_{ij} 
$$
and 
$$
\left(\tilde A^{(n,\rho)}\right)_{ij} = {1\over \sqrt \rho} 
\, a^{(n,\rho)}_{ij}, \quad i,j \in \{1, \dots, n\}.
$$
We have also denoted 
$$
\g^{(n,\rho_n)}(v) = {1- v^2/\rho\over  v} + {v(n-1)\over n} = 
v+{1\over v} -v \left( {1\over n} + {1\over \rho_n}\right). 
$$
The poles of $\tilde Z^{(n,\rho)}(v, \omega)$ \eqref{(2.6)}
with $v^2\neq \rho$ correspond to zeroes of the determinant of 
$H^{(n,\rho)}$. 
Introducing the subset 
\begin{equation}\label{(3.2)}
\Phi_{n,\rho} (v) = \{ \omega: \ \det(H^{(n,\rho)}(v)) = 0\},
\end{equation}
we can write that subsets $\Phi^{(i)}$ of \eqref{(2.8)} are  given by relations
$$
\Phi^{(i)}_{n,\rho}(\vep_n,\vep'_n) = \cup_{v \in D^{(i)}_{\vep_n, \vep'_n} }
\Phi_{n,\rho} (v),\quad i=1,2.
$$
It should be noted  that $\Phi^{(1)}$ and $\Phi^{(2)}$  are  determined as  uncountable unions of measurable events; since the probability space  
$\Omega_{n, \rho}$ 
generated by \eqref{(2.1)}  can be viewed  as a discrete set, we conclude that $\Phi^{(i)}$ are both measurable.
In what follows, we replace denotations of events 
$\{ \omega: \Upsilon (\omega)\} $  simply by $\{\Upsilon \}$. We will omit the superscripts
$n$ and $\rho$ in $H^{(n,\rho)}$ and everywhere below, when no confusion can arise.  We  prove relation \eqref{(2.10)} with $i=2$  in the full extent. 
The proof of (\ref{(2.10)}) in the case of  $i=1$ is given in less details.

\subsection{Proof of Theorem \ref{T1} in the case of $i=2$.}

The following elementary statement shows that the study of 
$\det(H(v))$ \eqref{(3.1)} can be reduced to the 
study of the product of singular values of $H(v)$. We denote 
these singular values by 
\begin{equation}\label{(3.3)}
\s_1(H(v))\le \s_2(H(v))\le \dots \le \s_n(H(v)).
\end{equation} 
Here and below we omit the superscripts $n$ and $\rho_n$ everywhere when no confusion can arise.

\v 
\n
 \begin{lemma}\label{L:3.1} For any $v\in \bC$, the equivalence
\begin{equation}\label{(3.4)}
\det(H(v))\neq 0 \iff \prod_{k=1}^n \s_k(H(v))\neq 0
\end{equation}
is true. 
\end{lemma}

\begin{proof}  Since $\tilde A $ is  a real symmetric matrix 
and $\hat B$ is a real diagonal one, then we can write for hermitian conjugate that 
$$
\big(H(v)\big)^*_{ij} = \overline{\big( H(v)\big)_{ji}}=
(1- \de_{ji} )\tilde A_{ji} - \bar v \de_{ji}  \hat B_{ii}  - \overline{\gamma(v)} = \big(H(\bar v)\big)_{ij}
$$
and therefore 
$$
H^*(v) = H(\bar v) \quad \hbox{and} \quad \overline{H(v)} = H(\bar v).
$$
It is easy to see that  $\l(v) $ is an eigenvalue of $H(v)$ if and only if 
 $\overline{\l(v)}$ is the  eigenvalue of $H(\bar v)$.
Then 
$$
\det(H(v))\neq 0 \iff\det (H(\bar v))\neq 0 \iff \det(H^*(v) H(v))\neq 0.
$$
The last statement is equal to that of the  right-hand side of  \eqref{(3.4)}. 
Let us note  that the last condition of \eqref{(3.4)} is equivalent to 
$\s_1(H(v))\neq 0$ because of \eqref{(3.3)} 
and due to positivity of $\sigma_i(H(v))$, $1\le i\le n$. 
Lemma \ref{L:3.1} is proved. 
\end{proof}

\v 
The study of singular values of $H(v)$ \eqref{(3.3)}
can be reduced to the study of singular values of $\tilde A- \g(v)I$ 
\eqref{(3.1)}  due to the concentration property \eqref{(2.15)} of the diagonal matrix 
$\hat B$.
Using the Weyl's inequality for singular values of $n$-dimensional matrices $X$ and $Y$  (see \cite{Cha} 
  and \cite{T}, Exercice 22),
\begin{equation}\label{(3.5)}
\vert \s_i(X+Y) - \s_i(X)\vert \le \Vert Y\Vert, \quad i=1, \dots, n,
\end{equation}
where $\Vert Y\Vert$ is the operator norm of $Y$,
we obtain that 
\begin{equation}\label{(3.6)}
\vert \s_i(H^{(n,\rho)}(v)) - \s_i(\tilde A^{(n,\rho)} - \g^{(n,\rho)}(v)I)\vert \le  \Vert v 
\hat B\Vert = \vert v \vert 
\max_{i=1, \dots, n} 
\vert \hat \D_i^{(n,\rho)}\vert , \quad i=1, \dots, n,
\end{equation}
where 
$$
\hD^{(n,\rho)}_i = 
  \tilde b^{(n,\rho)}_{ii} - {n-1\over n}= {1\over \rho} 
  \sum_{j\in \{1, \dots, n\}, \ j\neq i}\  \left( a^{(n,\rho)}_{ij} - {\rho\over n}\right).
$$
We denote 
\begin{equation}\label{(3.7)}
\hD_{\max}^{(n,\rho)} = \max_{i=1, \dots, n} \vert \hD_i^{(n,\rho)}\vert .
\end{equation}

We have seen above that  $\Phi(v)=\{ \s_1(H(v))=0\}$ and that 
\begin{equation}\label{(3.8)}
\{ \s_1(H(v))=0\}\subseteq 
\left
\{ \s_1(\tilde A - \gamma^{(n,\rho)}(v) I)\le | v| 
\hat \Delta_{\max}^{(n,\rho)}\right\}.
\end{equation}
Elementary calculation shows that
\begin{align}
\s_i(\tilde A - \gamma^{(n,\rho)}(v) I)    
    & = \l_i \big((\tilde A - \gamma^{(n,\rho)}(v) I)(\tilde A - \gamma^{(n,\rho)}(v) I)^*\big)
         \nonumber\\
    &= \l_i \big((\tilde A - \a(v) I)^2\big) + \b(v)^2,
\label{(3.9)}
\end{align}
where $\a(v)$ and $\b(v)$ are the real and imaginary parts of $\g(v)$, respectively. Diagonalizing  $\tilde A - \a(v) I$, we deduce 
from \eqref{(3.9)} that
\begin{equation}\label{(3.10)}
 \s_i(\tilde A - \gamma^{(n,\rho)}(v) I)= \big(\l_i(\tilde A) - \a(v)\big)^2 + \b(v)^2=
 |\l_i(\tilde A) - \g(v)|^2.
 \end{equation}

It follows from \eqref{(3.8)} and \eqref{(3.10)} 
that for any $v\in D^{(2)}_{\vep, \vep'}$, 
\begin{align}
\Phi_{n,\rho}(v)
&\subseteq 
\left\{ \min_{i=1, \dots, n} {|\l_i(\tilde A) - \g^{(n,\rho)}(v)|^2 \over |v| }
\le  \hat \Delta_{\max}^{(n,\rho)}\right\}
\nonumber\\
&\subseteq \left\{ \inf_{v \in D^{(2)}_{\vep, \vep'} }\ 
 \min_{i=1, \dots, n} {|\l_i(\tilde A) - \g^{(n,\rho)}(v)|^2 \over |v| }
\le  \hat \Delta_{\max}^{(n,\rho)}\right\}
= R^{(n,\rho)}(\vep, \vep'). 
\label{(3.11)}
\end{align}
 Then 
 \begin{equation}\label{(3.12)}
 \Phi^{(2)}_{n,\rho}(\vep, \vep') = \cup_{v \in D^{(2)}_{\vep, \vep'}} \{ \s_1(H(v))=0\}\subseteq
 R^{(n,\rho)}(\vep, \vep'),
 \end{equation}
 where we omitted $\vep$ and $\vep'$ in $R$. 
Let us note  that  
 the minimums of \eqref{(3.11)} 
 can be taken  in arbitrary order, in particular such that 
 $$
\inf_{v\in D^{(2)}_{\vep, \vep'}} = \inf _{1+\vep < r < 1/\vep'} \ \ \inf_{0\le \vp< 2\pi}.
$$
We denote these minimums by 
$
M_r^{(2)}$ and $ M_\vp
$
respectively.

In  \eqref{(3.11)}, the term  $|\l_i(\tilde A) - \g^{(n,\rho)}(v)|^2$ 
is a squared  distance between
real $\l_i(\tilde A)$ and the complex number
\begin{equation}\label{(3.13)}
\gamma^{(n,\rho)} (v) = v(1-\tau)+ {1\over v}, \quad \tau = 
{1\over n} + {1\over \rho}\, .
\end{equation}
Let us note that the application $\g^{(n,\rho)}(\cdot): \bC \to \bC$
 can be viewed as  a slightly modified version of the Zhukovsky transform
 $w(z)$ \eqref{(2.17)}.
For given $r$, the  points
$$
\CE(\dot a, \dot b)= \left\{\gamma^{(n,\rho)}(r e^{i\vp}), \quad 0\le \vp < 2\pi\right\}
$$
form an ellipsoid in the complex plane $\bR^2= \{(x,y): x= 
\Re w, \ y = \Im w\}$. 
An important property of the minimal distance from an ellipsoid to a given real point $\l$  
 is studied in Section \ref{S4}.

We  write that 
\begin{equation}\label{(3.14)}
R^{(n,\rho)}(\vep, \vep') \subseteq 
R^{(n,\rho)}_1 \cup R^{(n,\rho)}_2,
\end{equation}
where

$$
R^{(n,\rho)}_1= \left\{ M^{(2)}_{r} \, M_\vp \ 
 {|\l_1(\tilde A) - \g(re^{i\vp})|^2 \over r }
\le  \hat \Delta_{\max}^{(n,\rho)}\right\}
$$
and
$$
R^{(n,\rho)}_2= \left\{ M^{(2)}_{r}\,  M_{\vp} \ \min_{i=2, \dots, n}\
 {|\l_i(\tilde A) - \g(re^{i\vp})|^2 \over r }
\le  \hat \Delta_{\max}^{(n,\rho)}\right\}.
$$

 Let us study first the term $R^{(n,\rho)}_2$. To do this, we consider 
 the matrix 
\begin{equation}\label{(3.15)}
\breve A^{(n,\rho)} = {1- \delta_{ij}\over \sqrt{\rho} } \left(a^{(n,\rho)}_{ij} - {\rho\over n}\right).
\end{equation}
It is known \cite{FK} that  
\begin{equation}\label{(3.16)}
\l_1 (\bA) \ge \l_2(\tA) \ge \cdots \ge \l_n (\tA)\ge \l_n(\bA).
\end{equation}
For completeness, we reproduce in Section \ref{S4} the proof of \eqref{(3.16)}  by 
 F\"uredi and Koml\'os \cite{FK}.
It follows from \eqref{(3.16)} that  
\begin{equation}\label{(3.16')}
\min_{i=2, \dots, n}\
 {|\l_i(\tilde A) - \g(re^{i\vp})|^2 \over r }
\ge \min_{i=1, \dots, n}\
 {|\l_i(\breve A) - \g(re^{i\vp})|^2 \over r } .
 \nonumber 
\end{equation}
\v

Let us  
 introduce the set
 $$
 \U_\de = \{\omega: \l_{\max}
 (\breve A)\le 2+\de\}, \quad \de>0.
 $$
 Then we can write an obvious inclusion
 \begin{equation}\label{(3.17)}
R^{(n,\rho)}_2  \subseteq (R^{(n,\rho)}_2\cap \U_\de) \cup \bar \U_\de.
\end{equation}
We choose   $\delta$ and $\vep$   such that the point $2+\de$
lies inside the minimal ellipsoid,  
$$
2+\delta < \min_{1+\vep < r} \g(r),-
 $$ (see relation \eqref{(4.9)} of Section \ref{S4} for more details). In Section \ref{S4} we  show that (see Lemma \ref{L:4.2})
$$
M_\vp  {|\l_i(\breve A) - \g(re^{i\vp})|^2  } I_{\U_\de}\ge 
M_\vp  {|\l_{\max}(\breve A) - \g(re^{i\vp})|^2  } I_{\U_{\de}},
  $$
  where $I_\upsilon $ is  the indicator function of $\Upsilon$. 
It is also proved in Lemma \ref{L:4.2}    
that the minimal  distance between
$\lambda_{\max}(\breve  A)$ and  ellipsoid 
$\CE(\dot a, \dot b)$  is attained at the right extremum of 
$\CE(\dot a , \dot b)$ given by $\dot a = \g(r)$ (see \eqref{(4.7)}),
$$
M_\vp{|\l_{\max}(\breve A) - \g(re^{i\vp})|^2  }I_{\U_\de}
= {|\l_{\max}(\breve A) - \g(r)|^2  }I_{\U_\de}
\ge | (2+\de) - \g(r)|^2.
$$  
 Finally,  the minimal value 
 of  $| (2+\de) - \g(r)|^2/r$
 with respect to $r\in (1+\vep, 1/\vep')$  
 is given by the value
 $$
 \min_{1+\vep< r< 1/\vep'} { | (2+\de) - \g(r)|^2\over r}= 
 F^{(1)}(2+\de, \vep, \vep',\tau)= { (\vep^2 - \de(1+\vep) - \tau(1+\vep)^2)^2\over (1+\vep)^3}, 
 $$
 see Lemma \ref{L:4.2}. 
 Taking into account obvious inclusion, 
$$
R^{(n,\rho)}_2 \cap \U_\de \subseteq 
\left\{M_r^{(2)} \, M_\vp {|\l_{\max}(\breve A) - \g(re^{i\vp})|^2 \over r }I_{\U_\de} \le  \hat \Delta_{\max}^{(n,\rho)}
\right\},
$$
we can write that 
\begin{equation}\label{(3.18)}
P\left(
R^{(n,\rho)}_2 \cap \U_\de \right) \le P\left(\hat \Delta_{\max}^{(n,\rho)}
\ge F^{(1)}(2+\de, \vep, \vep', \tau)\right).
\end{equation}

It is proved in 
\cite{K-20} that if $\rho_n = \chi_n  \log n$ with 
$\chi_n\to \infty$ as $n\to\infty$, then for any positif $\nu$ the following upper bound holds,
\begin{equation}\label{(3.19)}
P\left( \hat  \Delta_{\max}^{(n,\rho_n)}> \nu \right)\le 
{1 \over 
n^{ \log(\nu \sqrt{ \chi_n}(1+o(1)))}}\ , \quad n\to \infty. 
\end{equation}
We see that in the limit \eqref{(2.5)} when $\chi_n$ infinitely increases, we can 
consider \eqref{(3.19)} with vanishing $\nu_n\to 0$ such that 
$
\nu_n \ge {3/ \sqrt{\chi_n}} .
$
 This allows us to choose $\vep$
in the right-hand side of \eqref{(3.18)}  such that $\vep\to 0$ as $n\to\infty$.
In this case   
 $F^{(1)}(2+\de,\vep,\vep', \tau)= \vep^4(1+o(1))$ 
 for vanishing $\de$ such that $\de = o(\vep^2)$ 
 (see relation \eqref{(4.13)} of Section \ref{S4}). 
 Then the choice of $\vep_n= 2 \chi_n^{-1/8}$ \eqref{(2.9)} is sufficient to conclude that 
\begin{equation}\label{(3.20)}
\sum_{n=1}^\infty P(R_2^{(n,\rho_n)}\cap \U_{\de_n}) < \infty,
\quad \de_n = o(\vep_n^2), 
\end{equation}
where $\rho_n$ satisfies conditions \eqref{(2.5)}. 
\v

Let us show that 
\begin{equation}\label{(3.21)}
\sum_{n\ge 1} P(\bar \U_{\de_n}) <\infty \quad {\hbox{with}} \quad \de_n = {1\over \chi_n^{3/8}}\ .
\end{equation}
It is easy to see that starting from some $n_0$,
random matrix $\bA^{(n,\rho_n)}$ \eqref{(3.15)} satisfies conditions of Theorem 2.7 of \cite{BGBK}.  According to this theorem,  there exists a constant $C>0$ such that 
\begin{equation}\label{(3.22)}
\bE \Vert \bA^{(n,\rho)}\Vert < 2+ {C\over \sqrt {\chi_n}},  
 \quad 
(\rho,\chi)_n\to \infty.
\end{equation}
 Adding to this result the general concentration inequality
\begin{equation}\label{(3.23)}
 P\left(\left\vert \Vert \bA^{(n,\rho)} \Vert -  \bE  \Vert \bA^{(n,\rho)} \Vert \right\vert \ge   t\right) \le
{2\over n^{ c \chi_n t^2}}, \quad t>0
  \end{equation}
  with some $c>0$  \cite{BLugM}, 
  we get the following asymptotic upper bound
 \begin{equation}\label{(3.24)}
   P\left( \Vert \bA \Vert \ge  2+ \de\right) 
   \le 
 P\left(\left\vert \Vert \bA \Vert -  \bE  \Vert \bA \Vert \right\vert 
 \ge \de - {C\over \sqrt{\chi_n}} \right)
 \le {2\over n^{ c \chi_n (\de/2)^2}}.
 \end{equation}
 It follows from \eqref{(3.22)} that 
$$
\bar \U_\de = \{ \l_{\max}(\breve A) > 2+\de\}
\subseteq \{ \l_{\max}(\breve A) - \bE \l_{\max}(\breve A) > 
\de -  C/\sqrt {\chi_n}\}.
$$
Then, using \eqref{(3.23)}, we get 
\begin{equation}\label{(3.26)}
P(\bar \U_\de) \le \left( | \l_{\max}(\breve A) - \bE \l_{\max}(\breve A)|
> \de -C/\sqrt {\chi_n} \right)
\le {2 n^{2\de c C \sqrt {\chi_n}}\over n^{c\chi_n \de^2 } \cdot n ^{cC^2}}.
\end{equation}
Relation \eqref{(3.26)} shows that  \eqref{(3.21)} is true, as well as \eqref{(3.20)}.
Then it follows from \eqref{(3.17)} that 
\begin{equation}\label{(3.27)}
\sum_{n\ge 1} P\left( R^{(n,\rho_n)}_2\right) < \infty.
\end{equation}

\v

Let us study the subset $R^{(n,\rho)}_1$  
\eqref{(3.14)}. 
We denote
\begin{equation}\label{(3.27)}
\Psi_\k = \left\{ \l_1(\tilde A) \ge \sqrt{\rho} (1-\k)\right\}
\end{equation}
and observe that 
\begin{equation}\label{(3.28)}
R_1^{(n,\rho)} \subseteq (R_1^{(n,\rho)} \cap \Psi_\k) \cup \bar \Psi_\k.
\end{equation}
Regarding the subset  $R_1^{(n,\rho)} \cap \Psi_\k$, we can repeat the previous reasoning based on the properties of an ellipsoid and write that
\begin{equation}\label{(3.29)}
R_1^{(n,\rho)}  \cap \Psi_\k \subseteq 
\left\{ 
\min_{1+\vep < r < 1/\vep'}\ 
{|\sqrt{\rho}(1-\k) - \g(r)|^2 \over r} 
\le 
\hat 
\Delta_{\max}^{(n,\rho)}
\right\}.
\end{equation}
With the help of the second part of Lemma 
\ref{L:4.2}, we deduce from \eqref{(3.29)} that 
\begin{equation}\label{(3.30)}
P(R_1^{(n,\rho)} \cap \Psi_\k) \le P\left( \hat 
\Delta_{\max}^{(n,\rho)}
\ge F^{(2)}(\sqrt \rho(1-\kappa), \vep,\vep', \tau)\right),
\end{equation}
where 
$$
F^{(2)}\left(\sqrt \rho(1-\kappa), \vep,\vep', \tau)\right)= 
 \vep' \left( (\vep')^2 (1-\tau) - q\vep' +1\right)^2, 
$$
under condition 
\begin{equation}\label{(3.31)}
1/\vep' < { q + \sqrt{ q^2 - 4(1-\tau)}\over 2},
\quad q = \sqrt\rho(1-\k).
\end{equation}
If $\vep'= (\sqrt{\rho}(1-h))^{-1}$, and $\k< h$ then 
\eqref{(3.31)} is verified and 
$$
F^{(2)}\left(\sqrt \rho(1-\kappa), \vep,\vep', \tau)\right),
= \sqrt\rho(h-\k)^2(1+o(1)).
$$
Substituting this relation into the right-hand side 
of \eqref{(3.30)}, 
we obtain the following upper bound, 
$$
P(R_1^{(n,\rho)} \cap \Psi_\k) \le P\left( \hat 
\Delta_{\max}^{(n,\rho)}
\ge \sqrt\rho (h-\k)^2\right).
$$
It follows from the last estimate and \eqref{(3.19)} that for any 
$h>\kappa$ we have 
$$
P(R_1^{(n,\rho)} \cap \Psi_\k) \le {1\over 
n^{\log(\sqrt{\chi_n \rho_n} (h-\kappa)^2(1+o(1))}}.
$$
Therefore we can write that 
\begin{equation}\label{(3.32)}
\sum_{n\ge 1} 
P(R_1^{(n,\rho)} \cap \Psi_\k)< \infty.
\end{equation}

\v
Let us estimate 
the probability of $\bar \Psi_\k$.
The maximal eigenvalue of the adjacency matrix $A^{(n,\rho_n)}$ 
\eqref{(2.1)} of the Erd\H os-R\'enyi random graphs $G(n,p_n)$ has been
studied by Krivelevich and Sudakov \cite{KS}. 
It follows from the results of \cite{KS} that with large probability,
the maximal eigenvalue  is greater than $(1+o(1))\max\left( \sqrt{D_{\max}}, np_n\right)$,
where $D_{\max}$ is the maximal vertex degree of the graph 
$\Gamma^{(n,\rho_n)}$, that is 
$$
P\left\{ \omega:\  \l_1\left(A^{(n,\rho_n)}(\omega)\right)\ge \max\left( 
\sqrt{\max_{i=1, \dots, n} b^{(n,\rho_n)}_{ii}}, \rho_n\right)(1+o(1))\right\} \to 1,
$$
in the limit $(\rho,\chi)_n\to \infty$ \eqref{(2.5)}. 
Then we can write that 
\begin{equation}\label{(3.33)}
\bE \l_1(\tA^{(n,\rho_n)})  \ge (1+o(1)) \sqrt {\rho_n}, \quad (\rho,\chi)_n\to\infty.
\end{equation}
In paper \cite{LMZ}, the following inequality 
\begin{equation}\label{(3.34)}
P\left( \sup_{\rho_n > C \log n} 
\left\vert  \l_1(\tA^{(n,\rho_n)}) - 
\bE \l_1(\tA^{(n,\rho_n)})\right\vert  
> {t\over \sqrt{\rho_n}}\right) \le 4 e^{-t^2/32}
\end{equation}
is proved for all $t>C$.  
Taking into account that $\rho_n = \chi _n \log n$ and choosing 
$t^2 = s^2 \rho_n$,  we deduce from \eqref{(3.34)} that 
\begin{equation}\label{(3.35)}
P\left( \left\vert  \l_1(\tA^{(n,\rho_n)}) - \bE \l_1(\tA^{(n,\rho_n )})\right\vert  >  s\right) \le {4\over n^{ \chi_n \, s^2/32 }}.
\end{equation}

It follows from \eqref{(3.33)} that 
$$
\bar \Psi_\k \subseteq 
\left\{ \l_{1}(\tilde A) \le \sqrt \rho(1-\k)\right\}
\subseteq \left\{ \bE \l_{1}(\tilde A)
-\l_{1}(\tilde A) \le \sqrt \rho\k (1+o(1))\right\}
$$
Then  \eqref{(3.35)} implies the following upper bound, 
$$
P\left(\bar \Psi_\k\right)\le 
P\left( |\bE \l_{1}(\tilde A)
-\l_{1}(\tilde A)| \ge \sqrt \rho\k (1+o(1))\right)\le 
{4\over n^{\chi_n \rho_n \k^2(1+o(1))/32}}.
$$
Then clearly 
$$
\sum_{n\ge 1} P(\bar \Psi_\kappa ) <\infty 
$$
for any $\kappa>0$.
Taking into account this convergence, as well as 
 \eqref{(3.28)}, \eqref{(3.32)} and \eqref{(3.35)}, 
 we conclude that 
$$
\sum_{n\ge 1} P\left( R^{(n,\rho_n)}_1\right)< \infty.
$$
Combining this relation with \eqref{(3.27)} and remembering 
\eqref{(3.12)}, we conclude that  
relation \eqref{(2.10)} is true in the case of $i=2$. 

\subsection{Proof of Theorem \ref{T1} in the case of $i=1$.}

We can write in analogy with \eqref{(3.11)} that
\begin{equation}\label{(3.36)}
\Phi(v)
\subseteq \left\{ \inf_{v \in D^{(1)}_{\vep, \vep'} }\ 
 \min_{i=1, \dots, n} {|\l_i(\tilde A) - \g^{(n,\rho)}(v)|^2 \over |v| }
\le  \hat \Delta_{\max}^{(n,\rho)}\right\}
= S^{(n,\rho)}(\vep, \vep'). 
\end{equation}
We introduce the subsets $S^{(n,\rho)}_1$ and $S^{(n,\rho)}_2$ similarly to \eqref{(3.14)} and \eqref{(3.15)}, where the 
minimum $M^{(2)}_r$
is replaced by 
$
M^{(1)}_r= \inf_{\vep'< r<1-\vep}
$.
Then we can write that
$$
S_2^{(n,\rho)} \cap \U_\de \subseteq 
\left\{M_r^{(1)} \, M_\vp {|\l_{\max}(\breve A) - \g(re^{i\vp})|^2 \over r }I_{\U_\de} \le  \hat \Delta_{\max}^{(n,\rho)}
\right\},
$$
and therefore
\begin{equation}\label{(3.37)}
P(
S^{(n,\rho)}_2 \cap \U_\de ) \le P(\hat \Delta_{\max}^{(n,\rho)}
\ge G^{(1)}(2+\de, \vep, \vep', \tau)).
\end{equation}
Relation \eqref{(4.19)} 
shows that $G^{(1)}(2+\de,\vep,\vep', \tau)= \vep^4(1+o(1))$ as $\vep \to 0$ 
and then  
\begin{equation}\label{(3.38)}
\sum_{n=1}^\infty P(S_2^{(n,\rho)}\cap \U_\de) < \infty
\end{equation}
with the same choice of $\vep_n = 2(\chi_n)^{-1/8}$ 
\eqref{(2.9)} and $\de_n = (\chi_n)^{-3/8}$ as in the previous subsection.

Let us study $S_1^{(n,\rho)}$.  Repeating the arguments of the previous
subsection, we can write that 
$$
S_1^{(n,\rho)} \cap \Psi_\k \subseteq 
\left\{ 
\min_{\vep' < r < 1-\vep}\ 
{|\sqrt{\rho}(1-\k) - \g(r)|^2 \over r} 
\le 
\hat 
\Delta_{\max}^{(n,\rho)}
\right\}
$$
and by Lemma \ref{L:4.3}, we have
\begin{equation}\label{(3.39)}
P(S_1^{(n,\rho)}\cap \Psi_\k) \le P\left( \hat 
\Delta_{\max}^{(n,\rho)}
\ge G^{(2)}(\sqrt \rho(1-\kappa), \vep,\vep', \tau)\right),
\end{equation}
Choosing $1/\vep'= \sqrt \rho(1-h)$,   
obtain we obtain with the help of \eqref{(4.20)}
  an upper estimate
$$
P(S_1^{(n,\rho)}\cap \Psi_\kappa)\le 
 P\left( \hat 
\Delta_{\max}^{(n,\rho)}
\ge \rho^{3/2}(1-h) (h-\k)^2\right).
$$
Then  \eqref{(3.19)} implies convergence 
\begin{equation}\label{(3.40)}
\sum_{n\ge 1} 
P(S_1^{(n,\rho)}\cap \Psi_\kappa)< \infty.
\end{equation}

 The upper bounds  for  the probabilities 
$P(\bar \U_\de)$ and $P(\bar \Psi_\kappa)$ 
in the case of $i=1$ are the same as in the case of $i=2$. 
Relations \eqref{(3.38)}, \eqref{(3.40)} together with 
\eqref{(3.21)} and \eqref{(3.35)} show that 
  convergence \eqref{(2.10)} is true  in the case of  $i=1$.
Theorem \ref{T1} is proved.

\subsection{The case of real $v$.}

Regarding the  particular case  $v\in \mathbb R$, we can use
 the following inequalities instead of \eqref{(3.6)},
\begin{equation}\label{(3.41)}
\vert \l_i(H^{(n,\rho)}(v)) - \l_i(\tilde A^{(n,\rho)} - \gamma^{(n,\rho)}(v)I) \vert \le  
|v | \, \hD^{(n,\rho)}_{\max}, 
\quad i=1, \dots, n, 
\end{equation}
where eigenvalues of $H^{(n,\rho)}$ and $\tilde A^{(n,\rho)}$ are ordered in, say, decreasing order. Relation 
\eqref{(3.41)} 
 is a corollary of more precise Weyl's inequality for hermitian matrices (see e.g. \cite{HJ}). 
Using \eqref{(3.41)},  we can write that
$$
\left\{ \omega: \det(H(v))=0\right\}=\cup_{i=1}^n 
\left\{ \omega: \l_i(H(v))=0\right\}
$$
and 
$$
\left\{ \omega: \l_i(H(v))=0\right\}
\subseteq
\left\{ {| \l_i (\tilde A) - \gamma(v)|\over \vert v \vert}
\le \hat \Delta^{(n,\rho)}_{\max}\right\}.
$$
Then, in complete analogy with \eqref{(3.2)}, \eqref{(3.11)} and \eqref{(3.12)},
we get inclusions
$$
\hat \Phi^{(2)}_{n,\rho}(\hat \vep, \vep')
= \cup_{1+\hat \vep < v<1/\vep'} 
\left\{ \omega: \det(H(v))=0\right\}
\subseteq \hat R^{(n,\rho)}_1\cup R^{(n,\rho)}_2,
$$
where 
$$
\hat R_1^{(n,\rho)}= \left\{ M_v^{(2)}  {| \l_1 (\tilde A) - \gamma(v)|\over \vert v \vert}
\le \hat \Delta^{(n,\rho)}_{\max}\right\}
$$
and 
$$
\hat R_2^{(n,\rho)}= \left\{ M_v^{(2)}  \min_{i=2, \dots, n} {| \l_i (\tilde A) - \gamma(v)|\over \vert v \vert}
\le \hat \Delta^{(n,\rho)}_{\max}\right\}
$$
with $M_v^{(2)}= \min_{1+\hat \vep< v< 1/\vep'}$. 
Regarding the last event,
we can repeat all computations of the previous subsection 
with $F^{(1)}(2+\delta, \vep, \vep', \tau)$ replaced by
$$
\hat F^{(1)}(2+\de, \hat \vep, \vep', \tau)= 
{ \vep^2 - \de(1+\vep) - \tau(1+\vep)^2\over (1+\vep)^3}
=O(\hat \vep^2), \quad \hat \vep\to 0,  
$$ 
in the asymptotic regime when $\delta = \hat \vep^2/h $ with sufficiently large  $h$. 
In this case relations \eqref{(3.18)} and \eqref{(3.19)} take the form
\begin{equation}\label{(3.42)}
P(\hat R^{(n,\rho)}_2 \cap \U_\de) \le P\left(\hat \Delta_{\max}^{(n,\rho)}
\ge \hat \vep^2
\right)
\le {1\over n^{ \log(\hat \vep^2 \sqrt{ \chi_n}(1+o(1)))}}\ , \quad n\to \infty. 
\end{equation}
It is easy to see that  \eqref{(2.11)} is sufficient for the convergence of the series \eqref{(3.42)}, 
\begin{equation}\label{(3.43)}
\sum_{n\ge 1} P\left( \hat R^{(n,\rho_n)}_2\cap \U_{\de_n}\right)< \infty,
\quad \de_n = \hat \vep_n^2/h.
\end{equation}
It follows from  the upper bound \eqref{(3.26)} that
to have the series of $P(\bar \U_{\de_n})$ convergent, we need to make
the value $\chi_n \de_n^2$ sufficiently large. This observation together with the last condition of \eqref{(3.43)} shows that 
the rate \eqref{(2.11)} is close to the optimal one from the technical point of view. 

\section{Auxiliary results and statements}\label{S4}

\subsection{Proof of Lemma 1.1 by Stark and Terras}

We reproduce here 
the proof of Lemma \ref{L:ST} given by Stark and Terras \cite{ST}. 
It is based on the observation 
that for finite $(q+1)$-regular graph relation (1.3) takes the form
\begin{equation}\label{(4.1)}
\left(Z_{X^{(q+1)}}(u)\right)^{-1} = (1-u^2)^{r-1} \prod_{j=1}^n \left(1 - u \l_j + qu^2\right),
\end{equation}
where $\l_1\le \dots\le \l_n$ are eigenvalues of $A$. 
Then the  poles of $Z_{X^{(q+1)}}(u)$ are given by zeros of 
quadratic polynomials $1-u \l_j + qu^2$, $1\le j\le n$. 

One can write 
$$
1 - u\l_j +q u^2 = (1- \a_j u) (1-\b_j u)
$$
with $\a_j \b_j = q $ and $\a_j+\b_j= \l_j$. Then 
$\a_j$ and $\b_j$ are given by the roots of the quadratic equation
$$
x^2 - \l_j x + q=0
$$
and therefore
\begin{equation}\label{(4.2)}
\a_j, \b_j = { \l_j + \sqrt{\l_j^2 - 4q}\over 2}.
\end{equation}
Thus the values $\a_j$ and $\b_j$ are complex conjugate
 if and only if 
\begin{equation}\label{(4.3)}
|\l_j|\le  2\sqrt q.
\end{equation}
and in this case  
$$
| \a_j|^2 = |\b_j|^2 = q.
$$ 
The last equalities mean that if $s= \s + i \tau$ is such that 
\begin{equation}\label{(4.4)}
q^s= \a_j\quad {\hbox{or}} \quad q^s= \b_j,
\end{equation}
then 
$Re(s)= \s=1/2$. 
If $|\l_j| = q+1$, then it follows from \eqref{(4.2)} and 
\eqref{(4.3)} that $Re(s)= 0$ or $1$. 
Finally, if $|\l_j|\neq q+1$, then $|\l_j|< q+1$ that therefore
\eqref{(4.4)} implies inequalities $0< Re(s)<1$. 
This argument completes the proof of equivalence between 
\eqref{(1.5)} and \eqref{(1.6)}. 

\subsection{Distance property of  ellipsoid.} 

Let ${\CE}(a,b)$ denote the family  of points on $\bR^2$ 
$$
\CE(a,b) = \left\{(s,t): s,t \in \bR, \ \left({s\over a}\right)^2
+ \left( { t\over b}\right)^2=1\right\}.
$$
We assume that $b<a$. 
 Given $x\in [0,a]$, we determine the distance $\CD(x)$ between the point 
 $(x,0)$ and the ellipse $\CE(a,b)$ by the formula 
 $$
 \CD(x)^2 = \min_{ (s,t) \in \CE(a,b)} ((s-x)^2 + t^2)= 
 \min_{ s\in [-a,a]} \phi(x,s),
 $$
 where 
 $$
 \phi(x,s)=  (s-x)^2 + b^2 - b^2 {s^2/ a^2}.
 $$

 \vskip 0.2 cm
 
\begin{lemma}\label{L:4.1} 	
There exists a critical point  $x_0= a(1-b^2/a^2)$ such that 
\begin{equation}\label{(4.5)}
 \CD(x)^2 = 
  \begin{cases}
  b^2 \left( 1 - {\displaystyle x^2\over \displaystyle a^2- b^2}\right) , & \text{if $0\le x \le x_0$}
   , \\
(a-x)^2, & \text {if $x_0 \le x \le a$ } 
\end{cases}
\end{equation}
\end{lemma}

\begin{proof}
The proof of Lemma \ref{L:4.1}  is based on the elementary analysis of the derivative
$$
{\partial \over \partial s} \phi(x,s) = 2(s-x) - 2b^2 s/a^2.
$$
If $x\in [0, x_0]$, then this derivative equals to zero  at the point 
$\tilde s = \tilde s(x)= x/(1-b^2/a^2)$ and 
$$
\CD(x)^2 = \phi(x,\tilde s(x)).
$$
If $x\in [x_0,a]$, then the derivative $\phi_s'(x,s)$ has no zero on the interval
$s\in [0,a]$ and therefore
$$
\CD(x)^2 = \phi(x,a)= (a-x)^2.
$$
This means that if $x<x_0$, then the corresponding point $\tilde l= (\tilde s,\tilde t)$ is such that $\tilde t>0$; if $x\ge x_0$, then the point $\tilde s$ coincides with the right extremity of the ellipsoid, $\tilde s=a$ for all $x\ge x_0$. 
The derivative of $\CD(x)^2$ is a discontinuous function and that the distance $\CD(x)^2$ shows a kind of "phase transition" of the first order at $x=x_0$. 
\end{proof}

\vskip 0.2cm

Let us point out that the distance  $\CD(x)$ (5.5) is a decreasing function for all $x\in [0, a]$ and for any $s\in [0,a]$,
\begin{equation}\label{(4.6)}
\min_{x\in [0, s]} \CD(x) = \CD(s).
\end{equation}

Now we turn to  the case of $\gamma^{(n,\rho)}(v)$ \eqref{(3.13)}. 
Denoting $s= Re(\g^{(n,\rho)}(v))$ and $t= Im(\g^{(n,\rho)}(v))$, 
we observe that the family of points 
$$
\dot \CE = \CE(\dot a, \dot b)= \left\{ \gamma^{(n,\rho)}(re^{i\vp}), \quad 0\le \vp< 2\pi\right\}
$$ 
is given by an ellipsoid with the half-axes
$$
\dot a = r(1-\tau) + {1\over r}, 
\quad \dot b= r(1-\tau)- {1\over r}.
$$
Regarding  the difference between $\l\in \bR$ and $\g^{(n,\rho)}$,
we see that its absolute value is bounded from below by the distance \eqref{(4.5)}
$$
| \l - \g^{(n,\rho)}|^2
\ge \CD(\l)^2
$$
with $a$ and $b$ replaced by $\dot a$ and $\dot b$, respectively.  
Taking into account that
$$
\dot x_0= {\dot a ^2 - \dot b^2\over \dot a}=
{4(1-\tau)\over r(1-\tau)+1/r},
$$
one can easily see that if
$$
\vep < {1} \quad {\hbox{and}} \quad  \tau < {1},
$$
then $\dot x_0< 2+\de$. 
Therefore we can write a version of \eqref{(4.6)}, 
\begin{equation}\label{(4.7)}
\min_{\l \in [0, 2+\de]} \ \min _{\vp\in [0, 2\pi)} | \l - \g^{(n,\rho)}|^2
= \left(r (1-\tau) + {1\over r} - 2 - \de\right)^2.
\end{equation}

\v 
Elementary analysis of the function
$$
\g(x)= x (1-\tau)+{1\over x}
$$
shows that if
\begin{equation}\label{(4.8)}
\vep > {1\over \sqrt{1-\tau}} -1 \quad {\hbox{and}} \quad 
\vep > \vep_0= {(2\tau+\de) + \sqrt{4\de + \de^2 
 +4\tau}\over 2(1-\tau)},
\end{equation}
then 
\begin{equation}\label{(4.9)}
\inf_{1+\vep < x} \g(x) >2+\delta.
\end{equation}

Regarding \eqref{(4.7)}, we see 
that it remains to study the minimal value of the function 
$$
f(x)={1\over x} { \left(x(1-\tau) + {1\over x} - q\right)^2}=
{ \big(x^2(1-\tau)  - qx +1\big)^2\over x^3}
$$ 
over the interval $x\in (1+\vep, 1/\vep')$
in the case when $q= 2+\delta$.
Elementary analysis shows that the  derivative  $f'(x)$ has four zeroes, 
$$
x_{1,2} = { q\mp \sqrt{q^2 - 4(1-\tau)}\over 2(1-\tau)}
\quad {\hbox{and}} \quad 
x_{3,4} = { -q \mp \sqrt{q^2 +12(1-\tau)}\over 2(1-\tau)}, 
$$
such that  $x_3< 0 < x_1< x_4<1$  and 
$$
1< x_2= {  q+ \sqrt{q^2 - 4(1-\tau)}\over 2(1-\tau)}.
$$
We will also need to minimize   $f(x)$ in the case when $q= \sqrt {\rho}(1 - \kappa)$.

\vskip 0.5cm
\begin{lemma}\label{L:4.2}
Let positive $\vep$ and $\vep'$ verify inequality
$1+\vep < 1/\vep'$. If $q$ is greater than $2-\tau$ and such that 
$x_2< 1+\vep$, then 
\begin{equation}\label{(4.10)}
F^{(1)}(q, \vep, \vep', \tau ) = \inf_{ 1+\vep<  x < 1/\vep'} f(x)= f(1+\vep)=
{ \left((1+\vep)^2(1-\tau) - q(1+\vep) +1\right)^2\over (1+\vep)^3}.
\end{equation}
If $q$ is greater than $2+\tau$ and such that  $1/\vep' < x_2$, then
\begin{equation}\label{(4.11)}
F^{(2)}(q, \vep, \vep', \tau )= \inf_{ 1+\vep<  x < 1/\vep'} f(x) =f(1/\vep')= 
{1\over\vep'} {\left(  (1-\tau) - q\vep' +(\vep')^2\right)^2}.
\end{equation}
	\end{lemma}

\begin{proof} Proof of Lemma \ref{L:4.2} is based on the  observation that
$f(x)$ has a local minimum at $x_2$ and is strictly 
decreasing on the interval $[1, x_2)$ and strictly increasing 
on the interval $(x_2, +\infty)$. Simple computations show that  \eqref{(4.10)} and \eqref{(4.11)} are true. 
\end{proof}

\v
Regarding  our main asymptotic regime when 
\begin{equation}\label{(4.12)}
\vep, \de, \tau \to 0, \quad \de = o( \vep^2), \quad 
{\hbox{and }}\quad \tau = o(\de),
\end{equation}
 we conclude that condition \eqref{(4.8)}
and the conditions of Lemma \ref{L:4.2}
 are satisfied and deduce from \eqref{(4.10)} that
asymptotic relation 
\begin{equation}\label{(4.13)}
F^{(1)}(2+\de , \vep , \vep', \tau) = \vep^4(1+o(1)), \quad \vep \to 0
\end{equation}
is true. 

\v 
Regarding \eqref{(4.11)} in the case when 
$q= \sqrt \rho(1-\kappa)$ and 
$1/\vep' = \sqrt \rho(1-h)$,
we conclude that in the limit of infinite $\rho$,  condition 
$\kappa< h$ is sufficient for inequality $1/\vep' < x_2$ to hold asymptotically. Simple computation shows that in this case  
\begin{equation}\label{(4.14)}
F^{(2)}(\sqrt \rho(1-\kappa), \vep, \vep', \tau)= 
\sqrt \rho\,  {(h-\kappa)^2} (1+o(1)), \quad \rho\to\infty.
\end{equation}

\vskip 0.5cm 
Let us study the minimum of $\g(x) $  over the interval $r\in (\vep', 1-\vep)$. The first observation is  that if
\begin{equation}\label{(4.15)}
\vep > { - (2\tau + \de) + \sqrt{\de^2 + 4\de + 4\tau}\over 2(1-\tau)},
\end{equation}
then the following analogue of \eqref{(4.9)} is verified,
\begin{equation}\label{(4.16)}
\min_{0<x < 1-\vep} \g(x) >2+\de.
\end{equation}

\begin{lemma}\label{L:4.3} Let positive $\vep$ and $\vep'$ verify inequality
$\vep'< 1-\vep$. If $q$ is greater than $2+\tau$ and 
such that $1-\vep < x_1$, then 
\begin{equation}\label{(4.17)}
G^{(1)}(q, \vep, \vep', \tau ) = \inf_{ \vep'<  x < 1-\vep} f(x)
= f(1-\vep)=
{ \left(\vep^2 - \de(1+\vep) - \tau(1+\vep)^2\right)^2\over (1-\vep)^3}.
\end{equation}
If $q$ is such that $x_1< \vep' $, then
\begin{equation}\label{(4.18)}
G^{(2)}(q, \vep, \vep', \tau )= \inf_{ \vep'<  x < 1-\vep} f(x) 
=f(\vep')
= 
{ \left( (\vep')^2 (1-\tau) - q\vep' +1\right)^2\over (\vep')^3}.
\end{equation}
\end{lemma}
\begin{proof}
	The proof of Lemma \ref{L:4.3} is based on elementary computations that we do not present here.
\end{proof}
 Regarding  the 
asymptotic regime \eqref{(4.12)}, we see that 
condition \eqref{(4.15)} 
and conditions of Lemma \ref{L:4.3} are  verified. Then 
we conclude that  
\begin{equation}\label{(4.19)}
G^{(1)}(2+\de , \vep , \vep', \tau) = \vep^4(1+o(1)), \quad \vep \to 0.
\end{equation}

\v 
 Regarding \eqref{(4.18)} in the case when 
$q= \sqrt \rho(1-\kappa)$ and 
$1/\vep' = \sqrt \rho(1-h)$,
we conclude that in the limit of infinite $\rho$,  condition 
$\kappa< h$ is sufficient for inequality 
$
\vep' > x_1
$
to hold asymptotically. 
Then relation
\begin{equation}\label{(4.20)}
G^{(2)}\left(\sqrt \rho(1-\kappa), \vep,\vep', \tau\right) = \rho^{3/2}(1-h)
(h-\kappa)^2(1+o(1))
\end{equation}
is true in the asymptotic regime \eqref{(4.12)}.

\subsection{Proof of F\" uredi-Koml\'os inequalities.} 
For completeness,  we reproduce the proof of inequalities 
\eqref{(3.16)}  resulting from  two lemmas below \cite{FK}.

\v 
\begin{lemma}\label{L:4.4} If $\tilde A=(a_{ij})$ is an $n\times n$
real symmetric matrix, and $ \breve A= \tilde A - tJ$ (where $J$ is the matrix with all $1$ entries), then
$$
\l_2(\tilde A)\le  \l_1(\breve A).
$$
\end{lemma}

\begin{proof} Relation 
$\l_1(\tilde A)= \max_{\Vert {\bx }\Vert=1} \bx \tilde A \bx $ and the Courant-Fischer theorem 
imply that 
$$
\l_2(\tilde A)= \min_{\bv} \ \max_{(\bx, \bv)=0, \, \Vert \bx \Vert =1} \bx \tilde A\bx,
$$
and therefore 
$$
\l_2(\tilde A)\le\  \max_{(\bx, {\bf 1})=0, \, \Vert \bx \Vert =1} \bx \tilde A\bx=
\ \max_{(\bx, {\bf 1})=0, \, \Vert \bx \Vert =1} \bx (\breve A+tJ)\bx
$$ 
$$
=\ \max_{(\bx, {\bf 1})=0, \, \Vert \bx \Vert =1} \bx \breve A\bx\le \l_1(\breve A),
$$
since $(\bx, {\bf 1})=0$ implies $J\bx = {\bf 0}$.  Lemma \ref{L:4.4}
is proved. 
\end{proof}

\begin{lemma}\label{L:4.5} If $\tilde A=(a_{ij})$ is a real symmetric matrix, and
$\breve A=\tilde A - tJ$, $t>0$, then
$$
\l_{-\infty}(\tilde A)\ge \l_{-\infty}(\breve A).
$$
\end{lemma}

\begin{proof} For $t>0$ the matrix $tJ$ is positive definite
(i.e. $\bx tJ\bx\ge 0$ for all $\bx \in {\bf R}^n$), hence
$$
\l_{-\infty}(\tilde A)=
 \min_{\Vert \bx \Vert=1} \bx \tilde A\bx \ge 
 \min_{\Vert \bx \Vert=1} \bx \breve A\bx 
+  \min_{\Vert \bx \Vert=1} \bx t J\bx 
$$
$$
\ge  \min_{\Vert \bx \Vert=1} \bx \breve A\bx 
= \l_{-\infty}(\breve A).
$$
Lemma \ref{L:4.5} is proved.
\end{proof}

\subsection*{Acknowledgments.} 

The author would like to thank the anonymous referee for valuable 
remarks and discussion that helped me much to improve the paper. 


\end{document}